\documentclass[10pt,a4paper]{article}
\usepackage{amsfonts}
\usepackage{amssymb,amsthm, amsmath,graphicx,mathtools}
\usepackage{bbm}
\usepackage{graphicx}

\usepackage{xcolor}

\usepackage{url}
\usepackage{multirow}

\newtheorem{theorem}{Theorem}

\newtheorem{proposition}[theorem]{Proposition}
\newtheorem{corollary}[theorem]{Corollary}
\newtheorem{lemma}[theorem]{Lemma}

\theoremstyle{remark}

\newtheorem{example}[theorem]{Example}

\def\CaP{\mathbf{P}}

\def\FraC{\mathcal{C}}
\def\FraF{\mathcal{F}}

\def\FraS{\mathcal{S}}
\def\FraP{\mathcal{P}}

\def\k{\mathbbmss{k}}
\def\N{\mathbb{N}}
\def\R{\mathbb{R}}
\def\Z{\mathbb{Z}}
\def\Q{\mathbb{Q}}
\def\d{\mathrm{d}}

\def\int{\mathrm{int} }

\title{Affine convex body semigroups and Buchsbaum rings}

\author{J. I. Garc\'{\i}a-Garc\'{\i}a\footnote{Departamento de Matem\'aticas, Universidad de C\'adiz,
E-11510 Puerto Real (C\'{a}diz, Spain). E-mail: ignacio.garcia@uca.es. Partially supported by the grant MTM2010-15595 and Junta de Andaluc\'{\i}a group FQM-366.}\\
A. S\'{a}nchez-R.-Navarro \footnote{Departamento Ingenier\'{\i}a Inform\'{a}tica, Universidad de C\'adiz,
E-11406 Jerez de la Frontera (C\'{a}diz, Spain). E-mail: alfredo.sanchez@uca.es. Partially supported by Junta de Andaluc\'{\i}a group FQM-366.}\\
A. Vigneron-Tenorio\footnote{Departamento de Matem\'aticas, Universidad de C\'adiz,
E-11406 Jerez de la Frontera (C\'{a}diz, Spain). E-mail: alberto.vigneron@uca.es. Partially supported by the grant MTM2012-36917-C03-01 and Junta de Andaluc\'{\i}a group FQM-366.}\\
}

\date{}

\begin{document}

\maketitle

\begin{abstract}

In this work, using the concept of convex body semigroup, we present new families of Buchsbaum.
We characterize Buchsbaum circle and convex polygonal semigroups and  we describe algorithmic methods to check such characterizations.


\smallskip
{\small \emph{MSC-class:} 20M14 (Primary), 20M05, 13H10 (Secondaries).}
\end{abstract}

\section*{Introduction}

Buchsbaum rings were introduced in the last half of the twentieth century and it has been treated from different points of view.
Two good introductions to Buchsbaum rings
are
 \cite{Goto} and \cite{Buchsbaum_rings}.
Given a field $\k$, $r$ indeterminates over it, $t_1,\ldots ,t_r$, and an affine cancellative commutative semigroup $S\subset \N^r$, the semigroup ring $\k[S]$
is defined as
the subring of $\k[t_1,\ldots ,t_r]$ generated by $t^s=t_1^{s_1}\cdots t_r^{s_{r}}$ with $s=(s_1,\ldots ,s_{r})\in S$. The semigroup $S$ is Buchsbaum if its associated semigroup ring $\k[S]$ is a Buchsbaum ring.
There are  many works devoted to the study of Buchsbaum affine semigroup rings (see for example \cite{Bresinsky}, \cite{BrunsGubeladzeTrung}, \cite{RosalesBuchs}, \cite{Kamoi}, \cite{Trung}, \cite{Buchsbaum_rings} and the references therein). A recurrent problem proposed in many of them is to find a criteria, expressed in terms of the affine semigroup $S$, to know if $\k[S]$ is Buchsbaum (see \cite{BrunsGubeladzeTrung}).

In this work we focus our attention on the study of Buchsbaum convex body semigroup rings.
Given $F\subset \R^r_\ge$ a non-empty convex body, we consider the so-called convex body semigroup $\FraF=\bigcup_{i=0}^{\infty} F_i\cap \N^r$ where $F_i=i\cdot F$ with $i\in \N$ (see \cite{ACBS} for further details).
This class of semigroups are useful to obtain examples of different kinds of rings.
For instance, in \cite{convex_CM_Go} the authors use these semigroups to characterize some families of
Cohen-Macaulay and Gorenstein rings families and they give computational methods to get examples.
All these characterizations are based on the easy method to check whether an element belong or not to a convex body semigroup.
In this work,
taking again advantage of this fact,
we characterize some families of Buchsbaum convex body semigroups
(Proposition \ref{circulos_Buchs} and Theorem \ref{theorem_polygon_Buchs}, respectively).
Moreover, we prove that these characterizations can be checked by using basic tools of Linear Algebra and Basic  Geometry and we use these tools
in the construction of Buchsbaum semigroup rings.
Besides, in Corollary \ref{triangulos_C-M} and \ref{rombos_C-M}, we give explicitly families of Buchsbaum semigroups. 
In this work, we also introduce the {\tt Mathematica} package {\tt PolySGTools} (\cite{programa_poligonos}). This package is used to compute the minimal generating set of a convex polygonal semigroup given by a rational polygon and to check
Buchsbaumness of an affine convex polygonal semigroup.

The contents of this paper are organized as follows.
In Section \ref{sec1}, we provide some basic tools and definitions that are used in the rest of the work. In Section \ref{sec2}, Buchsbaum affine circle semigroups are characterized. Finally, Section \ref{sec3} is devoted to the study of
 properties that characterize  Buchsbaum affine convex polygonal semigroups (Theorem \ref{theorem_polygon_Buchs}) and to give explicit families of Buchsbaum affine convex polygonal semigroups.

\section{Preliminaries}\label{sec1}

For any $L$ subset of $\R^r$, denote by $L_{\ge}$ the set $\{(x_1,\dots ,x_r)\in L|x_i\ge 0,\,  i=1,\dots,r \}$. Let $G$ be a non-empty subset of $\R^r_\ge$, denote by $L_{\Q_\geq}(G)$ the cone $\left\{\sum_{i=1}^p q_if_i| p\in\N, q_i\in \Q_{\geq}, f_i\in G \right\}$. Let  $Fr(L_{\Q_\geq}(G))$ be the boundary of $L_{\Q_\geq}(G)$ considering the usual topology of $\R^r$ and define the interior of $G$ as $G\setminus Fr(L_{\Q_\geq}(G))$, denote it by $\int(G)$. We use $\d(P)$ to represent the Euclidean distance from a point $P$ to the origin $O$.

Let $S\subset \N^r$ be the affine semigroup generated by $\{n_1,\ldots ,n_r,n_{r+1},\ldots ,n_{r+m}\}$.
A semigroup $S$ is called simplicial if
$L_{\Q_\geq}(S)=L_{\Q_\geq}(\{n_1,\ldots ,n_r\})$. All semigroups appearing in this work are simplicial, so in the sequel we will assume such property. In the case the semigroup ring $\k[S]$ is a Cohen-Macaulay ring, $S$ is Cohen-Macaulay semigroup.
The set  $L_{\Q_\geq}(S)\cap \N^r$ is an affine semigroup which is denoted by $\FraC$.

Let $\overline{S}$ be the semigroup $\{a\in\N^{r}| a+n_i\in S,\, \forall i=1,\ldots ,r+m\}$. It is straightforward to prove that  $\overline{S}\subset \FraC.$ The following result shows a characterization of Buchsbaum rings in terms of their associated semigroups.

\begin{theorem}\cite[Theorem 5]{RosalesBuchs}\label{RosalesBuchs}
The following conditions are equivalent:
\begin{enumerate}
\item $S$  is Buchsbaum.
\item $\overline{S}$ is Cohen-Macaulay.
\end{enumerate}
\end{theorem}

In \cite[Theorem 9]{RosalesBuchs}, it is given a method to check if a simplicial semigroup is Buchsbaum. To apply such method it is necessary to compute the intersection of the Ap\'{e}ry sets of the generators of the rational cone of $S$ (the elements $n_1,\ldots,n_r$). Such intersection is computed using the method presented in \cite{RosalesCM} that uses some bounds to describe a region where the elements of the Ap\'{e}ry set are; the high value of the bound obtained makes the algorithm impractical in many cases.
Thus, to determine if $S$ is Buchsbaum, it is  necessary to check if the semigroup $\overline{S}$ is Cohen-Macaulay.
In this work we focus on solving algorithmically this problem for some kinds of subsemigroups of $\N^2.$

Given $S\subseteq \N^2$ an affine semigroup, denote by $\tau_1$ and $\tau_2$ the extremal rays of $\FraC=L_{\Q_\geq}(S)\cap \N^2$ with $\tau_1$ the ray with greater slope and by $n_1\in \tau_1$ the element of $S\cap \tau_1$ with less module, similarly define $n_2\in\tau_2$. Note that $\FraC\cap \tau_j=\N^2\cap \tau_j$ (with $j=1,2$) is a subsemigroup of $\N^2$ and that it is generated only by an element.

\begin{corollary}\cite[Corollary 2]{convex_CM_Go}\label{C-M}
Let $S\subseteq \N^2$,
the following conditions are equivalent:
\begin{enumerate}
\item $S$ is Cohen-Macaulay.
\item For all $a\in\FraC\setminus S$, $a+n_1$ or $a+n_2$ does not belong to $S$.
\end{enumerate}
\end{corollary}

\begin{lemma}\cite[Lemma 3]{convex_CM_Go}\label{lemma_no_C-M}
Let $S\subseteq \N^2$ be a simplicial affine semigroup such that $\int (\FraC) \setminus \int(S)$ is a non-empty finite set, then $S$ is not Cohen-Macaulay.
\end{lemma}

From now on, we consider only semigroups associated to convex bodies.
Let $F\subset \R^r_\ge$ be a non-empty convex body,
 the convex body semigroup generated by $F$ is the semigroup $\FraF=\bigcup_{i=0}^{\infty} F_i\cap \N^r.$ In general, these semigroups are not finitely generated. An interesting property of  them is that it is easy to check if an element $P$ belongs to a given semigroup.
Just proceed as follows:
take $\tau$ the ray defined by $P$ and the segment $\tau\cap F=\overline{AB}$ with $\d(A)\leq \d(B)$; the element $P$ belongs to $\FraF$ if and only if the set $\{k\in\N|\frac{\d(P)}{\d(B)}\leq k \leq \frac{\d(P)}{\d(A)}\}$ is non-empty.

In \cite{ACBS},  affine convex body semigroups are characterized when the initial convex body is a circle or a convex polygon. In both cases, a convex body semigroup is affine if and only if the intersection of the initial convex body with each extremal ray of its associated positive integer cone contains at least a rational point. Besides, the minimal system of generators of these convex body semigroups can be computed algorithmically (see Theorem 14 and Theorem 18 in \cite{ACBS} for further details). Let $F\subset \R^2$ be a convex body, in this case the positive integer cone $\FraC$ is equal to  $L_{\Q_\geq}(F\cap \R ^2_{\geq})\cap \N^2$
and  $\int(\FraC)=\FraC\setminus\{\tau_1,\tau_2\}.$

\section{Buchsbaum affine circle semigroups}\label{sec2}

Let $C\subset\R^2$ be the circle with center  $(a,b)$ and radius $r>0$ with $a,b,r\in\R$; define  $C_i$   the circle with center $(ia,ib)$ and radius $ir$ and $\FraS=\bigcup_{i=0}^\infty C_i\cap \N^2$ the so-called circle semigroup associated to $C$.
Note that
when $C\cap \R^2_\geq$ has at least two points the circle semigroup is simplicial, so
in this section we consider that $\FraS$ is always a simplicial affine circle semigroup.
Let $\overline{\FraS}$ be the semigroup $\{a\in\N^{2}| a+n_i\in \FraS,\, \forall i=1,\ldots ,m\}$ with $\{n_1,n_2,\ldots ,n_{m}\}$ the minimal system of generators of $\FraS.$

\begin{proposition}\label{circulos_Buchs}
Let  $\FraS\subset \N^2$ be an affine circle semigroup. The semigroup $\FraS$ is Buchsbaum if and only if $\int(\FraC)=\int(\overline{\FraS})$ and $\overline{\FraS}\cap \tau_j$ is generated only by one element for $j=1,2$.
\end{proposition}

\begin{proof}
By Theorem \ref{RosalesBuchs},
 $\FraS$ is Buchsbaum if and only if $\overline{\FraS}$ is Cohen-Macaulay. We prove that $\overline{\FraS}$ is Cohen-Macaulay if and only if $\int(\FraC)=\int(\overline{\FraS})$ and $\overline{\FraS}\cap \tau_j$ is generated by only one element for $j=1,2$.

Assume that $\overline{\FraS}$ is Cohen-Macaulay and suppose that $\int(\FraC)\setminus \int(\overline{\FraS})\neq \emptyset$.
Let $n_j'$ be one element of the minimal system of generators of $\overline{\FraS}\cap \tau_j$ with $j=1,2.$
Since there exists a real number $d>0$ such that $\{a\in\int(\FraC)|\d(a)>d\}\subset \FraS$ (see \cite[Lemma 17]{ACBS}),
the set $\int(\FraC)\setminus\int(\FraS)$ is finite, and thus $\int(\FraC)\setminus\int(\overline{\FraS})$ is finite too.
Take $a\in \int(\FraC)\setminus \int(\overline{\FraS})$ verifying that  $\d(a)=\max \{\d(a')| a'\in \int (\FraC) \setminus \int(\overline{\FraS})\}$. The elements $a+n_1'$ and $a+n_2'$ are in $\overline{\FraS}$ and by Corollary \ref{C-M} the semigroup $\overline{\FraS}$ is not Cohen-Macaulay
which is a contradiction.
Let us prove now that $\overline{\FraS}\cap \tau_j$ is generated by only one element for $j=1,2$. We consider two different cases: ${\FraS}\cap \tau_j$ is generated only by one element or not. If there exist $n_j\in\N^2$ such that ${\FraS}\cap \tau_j=\langle n_j\rangle$ for $j\in\{1,2\}$, then
for every $a\in (\tau_j\setminus \FraS)\cap\N^2$ we have that $a+n_j\in \tau_j\setminus S$ and hence
$\overline{\FraS}\cap \tau_j=\FraS\cap \tau_j.$
We consider now the case that ${\FraS}\cap \tau_j$ is minimally generated by two or more elements with $j\in\{1,2\}$.
We have that $C\cap \tau_j$ is a segment and that $(\FraC\setminus \FraS)\cap \tau_j$ is a finite non-empty set.
This implies that
$(\FraC\setminus \overline{\FraS})\cap \tau_j$ is a finite set too.
If this is non-empty, take the element $a\in (\FraC\setminus \overline{\FraS})\cap \tau_j$ such that $\d(a)=\max \{\d(a')| a'\in (\FraC\setminus \overline{\FraS})\cap \tau_j\}$. It verifies that $a+n_1'$ and $a+n_2'$ belong to $\overline{\FraS}$ and therefore $\overline{\FraS}$ is not Cohen-Macaulay (Corollary \ref{C-M}).

Assume now that $\int(\FraC)=\int(\overline{\FraS})$ and that $\overline{\FraS}\cap \tau_j$ is generated by only one element for $j=1,2$. In this case, it is straightforward to prove that for all $a\in \FraC\setminus \overline{\FraS}$, $a+n_1\not\in\overline{\FraS}$ or $a+n_2\not\in\overline{\FraS}$. By Corollary \ref{C-M}, $\overline{\FraS}$ is Cohen-Macaulay.
\end{proof}

Using the above proof, the conditions of Proposition \ref{circulos_Buchs} can be determined from the initial circle. To check whether $\int(\FraC)=\int(\overline{\FraS})$, we only have to compute the finite set $\int (\FraC) \setminus \int(\overline{\FraS})$ by using the bound provided by \cite[Lemma 17]{ACBS}. The second condition
is satisfied whether $C\cap \tau_j$ is a point or, in the case $C\cap \tau_j$ is a segment, if the generator of $\FraC\cap \tau_j$ belongs to $\overline{\FraS}.$ Both conditions can be checked algorithmically.

\begin{example}
Let $C$ be the circle with center $(7/5,4/5)$ and radius $1/5$. Using the program {\texttt{CircleSG}} (see \cite{programa}), we obtain that the affine circle semigroup\footnote{Note that $C\cap \tau_1=(32/25, 24/25)$ and $C\cap \tau_2=(96/65, 8/13).$}
 $\FraS$
associated to $C$ is minimally generated by the set $$\Big\{(4,2),(5,3),(6,3),(6,4),(7,3),(7,4),(7,5),(8,5),(9,4),(9,6),(10,7),$$
$$(11,8),(15,11),(19,8),(19,14),(23,17),(27,20),(31,13),(31,23),(32,24),$$
$$(35,26),(43,18),(55,23),(67,28),(79,33),(91,38),(96,40)\Big\}$$ and $\int(\FraC)\setminus\int({\FraS})$ is $\{(2,1),(3,2)\}$ (see Figure \ref{ejemplo_circle}).
\begin{figure}[h]
    \begin{center}
\begin{tabular}{|c|}\hline
\includegraphics[scale=0.4]{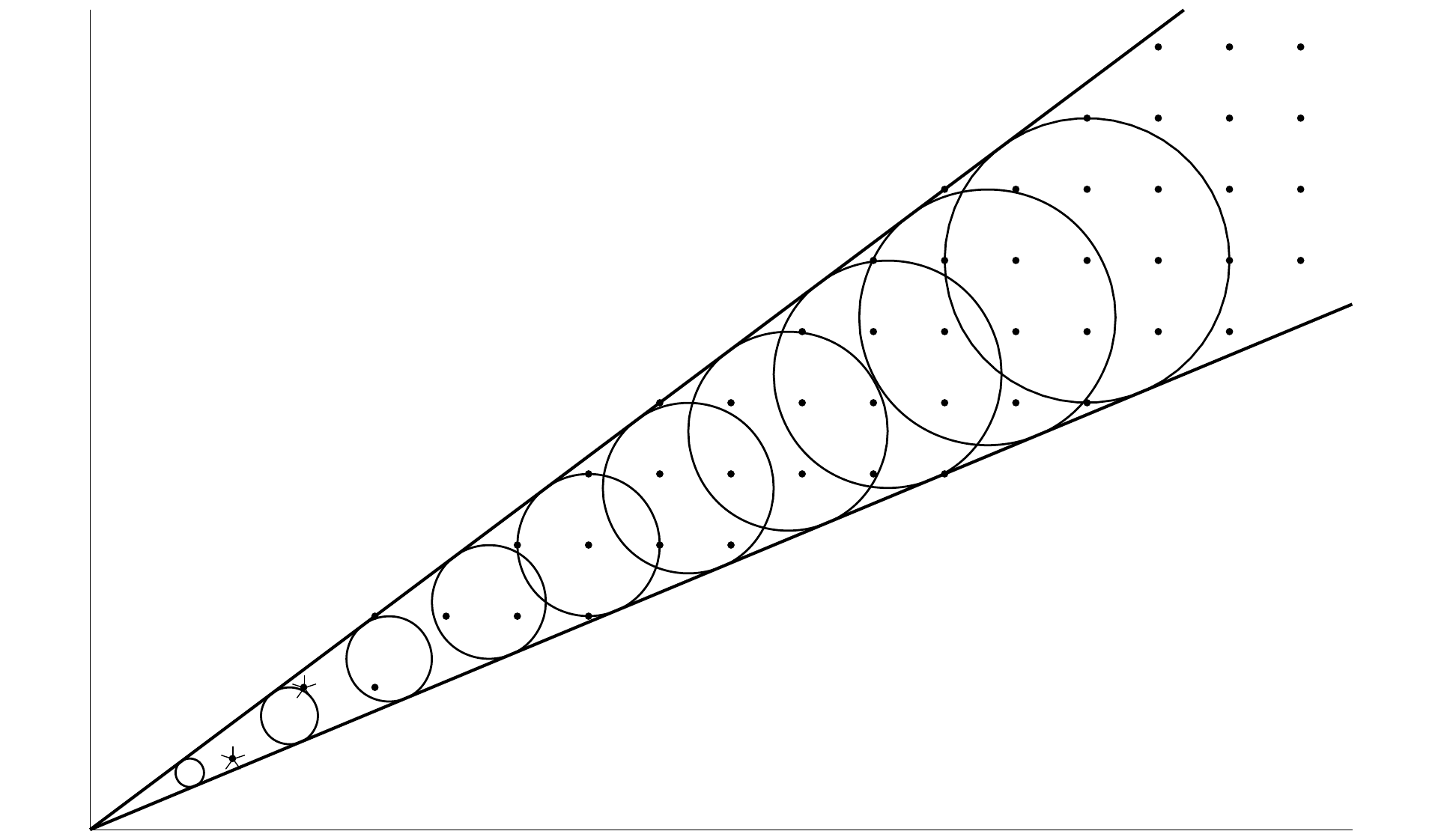} \\ \hline
\end{tabular}
\caption{The affine circle semigroup associated to the circle with center $(7/5,4/5)$ and radius $1/5$.}\label{ejemplo_circle}
\end{center}
\end{figure}
It is easy to check that  the points $(2,1)$ and $(3,2)$ belong to $\overline{\FraS}$. Thus, $\int(\FraC)=\int(\overline{\FraS}).$ Besides, $\overline{\FraS}\cap \tau_1=\langle (32,24) \rangle$ and $\overline{\FraS}\cap \tau_2=\langle (96,40) \rangle$.
By Proposition \ref{circulos_Buchs}, the affine circle semigroup $\FraS$ is Buchsbaum.
\end{example}

For any affine semigroup, a problem of a high computational complexity is the problem of determining whether an element belongs to it.
As explained above, in the particular case of circle semigroups this problem is simple.
This fact simplify
the computation of the above example and allows to obtain the result very quickly.

\section{Buschbaum affine convex polygonal semigroups}\label{sec3}

Denote by  $F\subset \R^2_{\geq}$ a  compact convex polygon (not equal to a segment) with vertex set $\CaP=\{P_1,\ldots ,P_t\}$ arranged in counterclockwise direction and let $\FraP=\bigcup_{i=0}^{\infty} F_i\cap \N^2$ be its associated semigroup. Note that since the fixed convex polygon $F$ is not a segment, $\FraP$ is a simplicial affine convex polygonal semigroup. As in previous sections,  $\tau_1$ and $\tau_2$ are the extremal rays of $\FraC$  assuming $\tau_1$ with a slope greater than the slope of $\tau_2$. Let $\overline{\FraP}$ be the semigroup $\{a\in\N^{2}| a+n_i\in \FraP,\, \forall i=1,\ldots ,m\}$ with $\{n_1,n_2,\ldots ,n_{m}\}$ the minimal system of generators of $\FraP$ and let $n_j'$ be a minimal generator of $\overline{\FraP}\cap \tau_j$ with $j=1,2.$

In order to prove the results of this section, we consider different special subsets of the cone $\FraC$ and some points and lines in $L_{\Q_\geq}(F)$. We distinguish two cases, $F\cap \tau_i$ is a point or a segment.

Assume $F\cap \tau_1=\{P_1\}\subset \CaP$, let $j$ be the least positive integer
such that $j\overline{P_{1}P_{t}} \cap (j+1)\overline{P_{1}P_{2}}$ is not empty.
Since $\overline{P_{1}P_{t}}$ and $\overline{P_{1}P_{2}}$ are not parallel,
there exists a point $\{V_1\}=j\overline{P_{1}P_{t}} \cap (j+1)\overline{P_{1}P_{2}}$
(using \cite[Lemma 11]{ACBS}, $V_1$  can be easily computed).
Denote by $T_1$ the triangle with vertex set $\{ O, P_1, V_1-jP_1 \},$ and by $\stackrel{\circ}{T_1}$ its topological interior. By \cite[Lemma 11]{ACBS}, for every $h\in\N$ with $h\geq j$
the points $h\overline{P_{1}P_{t}} \cap (h+1)\overline{P_{1}P_{2}}$ are in the same straight line, which we denote by $\nu_1$.
Note that
$((\stackrel{\circ}{T_1}\cup (\overline{OP_1}\setminus\{O,P_1\}))+\mu P_1)\cap\FraP=\emptyset$ for all $\mu\in\Z_{\geq}$.
This construction allows us to define the set
$$\mathcal{B}_1=\{ D+\lambda n_1 | D\in \overline{(jP_1)V_1} \text{ and } \lambda\in \Q_{\ge} \}\cap\FraC$$ whose elements are in $\FraP$ or  they are in $\bigcup_{\mu\in \N,\, \mu\ge j}\big( (\stackrel{\circ}{T_1}\cup (\overline{OP_1}\setminus\{O,P_1\}))+\mu P_1\big)$.
The elements of $\mathcal{B}_1$ verify that
if $P\in \mathcal{B}_1\setminus \FraP$ then $P+n_1\not\in \FraP$ and thus $P\notin \overline{\FraP}$; this  implies that  $\FraP\cap \mathcal{B}_1= \overline{\FraP}\cap \mathcal{B}_1$.
Denote by $\Upsilon_1$ the finite set ${\rm ConvexHull} (\{O,jP_1,V_1, \nu_1\cap \tau_2 \}) \cap \N^2.$
Analogously,
if the set $F\cap \tau_2=\{P_1\}\subset \CaP$
(for the sake of simplicity, we  call again this point  $P_1$)
there exists the least integer $j$ such that $j\overline{P_{1}P_{2}} \cap (j+1)\overline{P_{1}P_{t}}$ is equal to $\{V_2\}$. Let  $T_2$ be the triangle with vertex set $\{ O, P_1, V_2-jP_1 \},$ and denote by $\nu_2$ the line containing the points $\{h\overline{P_{1}P_{2}} \cap (h+1)\overline{P_{1}P_{t}}| h\ge j, h\in\N\}$
and by $\mathcal{B}_2$ the set $\{ D+\lambda n_2 | D\in \overline{(jP_1)V_2} \text{ and } \lambda\in \Q_{\ge} \}\cap \FraC$.
All of the properties of these sets are analogous to the properties of the sets defined previously for $\tau_1$.
Denote by $\Upsilon_2$  the finite set ${\rm ConvexHull} (\{O,jP_1,V_2, \nu_2\cap \tau_1 \}) \cap \N^2.$
In the case
 $F\cap \tau_i$ is a segment for some  $i$, we take  $\nu_i=\tau_i$ and $\Upsilon_i=\{O\}$.

We
 define the set $\Upsilon = (Q+\L_{\Q_\geq}(F))\cap \N^2\subset \FraC$ with
 $\{Q\}=\nu_1\cap \nu_2\subset \L_{\Q_\geq}(F)$.
Note that the boundary of the set $\Upsilon$ intersects with two different sides of the polygon  $i_0F$ when $i_0\gg 0$ and therefore the sets  $\Upsilon\setminus \FraP$ and $\Upsilon\setminus \overline{\FraP}$ are finite.
The last set we define is the finite set
$\Upsilon '=\{ a\in (\Upsilon _1\cup \Upsilon _2)\setminus \overline{\FraP}|\, a+n_1',a+n_2'\in \overline{\FraP} \}.$
It is straightforward to prove that the cone $\FraC$ is the union of
$\mathcal{B}_1,$ $\mathcal{B}_2,$ $\Upsilon _1,$ $\Upsilon _2$ and $\Upsilon.$

\begin{theorem}\label{theorem_polygon_Buchs}
Let $\FraP$ be a simplicial affine convex polygonal semigroup. Then
\begin{enumerate}
\item \label{caso1} if $\int (\FraC) = \int(\overline{\FraP})$, the semigroup $\FraP$ is Buchsbaum if and only if $\overline{\FraP}\cap \tau_j$ is generated by only one element for $j=1,2$,
\item \label{caso2} if $\int (\FraC) \neq \int(\overline{\FraP}),$ the semigroup $\FraP$ is Buchsbaum if and only if $\Upsilon'=\emptyset$ and $\Upsilon\subset \overline{\FraP}.$
\end{enumerate}
\end{theorem}

\begin{proof}

We prove that $\overline{\FraP}$ is Cohen-Macaulay if and only if the conditions of the theorem are fulfilled.
Due to the similarity with the proof of Proposition \ref{circulos_Buchs}, the case \ref{caso1} is left to the reader.

Assume that $\int (\FraC) \neq \int(\overline{\FraP})$ and that $\overline{\FraP}$ is Cohen-Macaulay. By Corollary \ref{C-M}, the set $\Upsilon '$ has to be empty. If $\Upsilon\not \subset \overline{\FraP}$,  choose $a\in \Upsilon\setminus \overline{\FraP}$ such that $\d(a)=\max \{\d(a')| a'\in \Upsilon\setminus \overline{\FraP}\}$. Then $a+n_1'$ and $a+n_2'$ belong to $\overline{\FraP}$ which implies that  $\overline{\FraP}$ is not Cohen-Macaulay. Thus $\Upsilon\subset \overline{\FraP}.$

Conversely, let $a$ be an element of $\FraC\setminus \overline{\FraP}$ (note that $a\notin \Upsilon\subset \overline{\FraP}$).
We discuss the possibilities we have.
  If $F\cap \tau_1$ is a point and $a$ belongs to the strip bounded by the parallel lines $\tau_1$ and $\nu_1$, we have that $\FraP\cap \tau_1=\overline{\FraP}\cap \tau_1$ and $n_1=n_1'$.
 Besides, the element  $a$ belongs to $\Upsilon _1\setminus \overline{\FraP}$ or it belongs to $\mathcal{B}_1\setminus \overline{\FraP}$. Since $\Upsilon'=\emptyset$, if
  $a\in \Upsilon _1\setminus \overline{\FraP}$,  the element $a+n_1'$ or $a+n_2'$ does not belong to $\overline{\FraP}$  and if $a\in \mathcal{B}_1\setminus \overline{\FraP}$  then $a+n_1'\notin \overline{\FraP}.$ We proceed similarly in the case of $F\cap \tau_2$ is a point and $a$ belongs to the strip bounded by the parallel lines $\tau_2$ and $\nu_2$, obtaining that the element $a+n_1'$ or $a+n_2'$ does not belong to $\overline{\FraP}$. If $F\cap \tau_2$ is a segment, since $\Upsilon \subset \overline{\FraP},$ we have that  $\FraC\setminus \overline{\FraP}\subset\Upsilon_1\cup \mathcal{B}_1$, and thus every $a\in\FraC\setminus \overline{\FraP}$ verifies  $a+n_1'$ or $a+n_2'$ is not in $\overline{\FraP}$.
  Similarly, when $F\cap \tau_1$ is a segment and  $F\cap \tau_2$ is a single point, for every  $a\in \FraC\setminus\overline{\FraP}$ we obtain again that $a+n_1'$ or $a+n_2'$ is not in $\overline{\FraP}$.
  In any of the above cases, every element $a\in\FraC\setminus \overline{\FraP}$  fulfills that at least one element of the set $\{a+n_1',a+n_2'\}$ does not belong to $\overline{\FraP}$, and hence $\overline{\FraP}$ is Cohen-Macaulay (Corollary \ref{C-M}).
  Finally, if $F\cap \tau_1$ and $F\cap \tau_2$ are both segments, then $\FraC=\Upsilon\subset \overline{\FraP}$; this implies $\int (\FraC) = \int(\overline{\FraP})$, which is a contradiction.
\end{proof}

As in the circle semigroup case, to apply the above result it is necessary to
check whether $\int (\FraC)=\int(\overline{\FraP})$.
 The different situations are the following:

\begin{enumerate}
\item If $F\cap \tau_1$ is a segment $\overline{P_1P_t}$ and $F\cap \tau_2$ is a segment $\overline{P_{d-1}P_d}$, the set $\Upsilon$ is equal to the positive integer cone $\FraC$ and the sets $\FraC\setminus \FraP$ and $\FraC\setminus \overline{\FraP}$ are finite. Let $j\in \N$ be the least integer such that $j\overline{P_1P_t}\cap (j+1)\overline{P_1P_t}\neq \emptyset$ and $j\overline{P_{d-1}P_d}\cap (j+1)\overline{P_{d-1}P_d}\neq \emptyset$, and let $T$ be the triangle with vertex set $\{O,jP_1,jP_{d-1}\}.$ Clearly, $T\cap \N^2$ is finite and
    $\int (\FraC)\setminus\int(\overline{\FraP})\subseteq T\cap \N^2$.
    This is illustrated in Figure \ref{caso_2_segmentos}.
\begin{figure}[h]
    \begin{center}
\begin{tabular}{|c|}\hline
\includegraphics[scale=0.42]{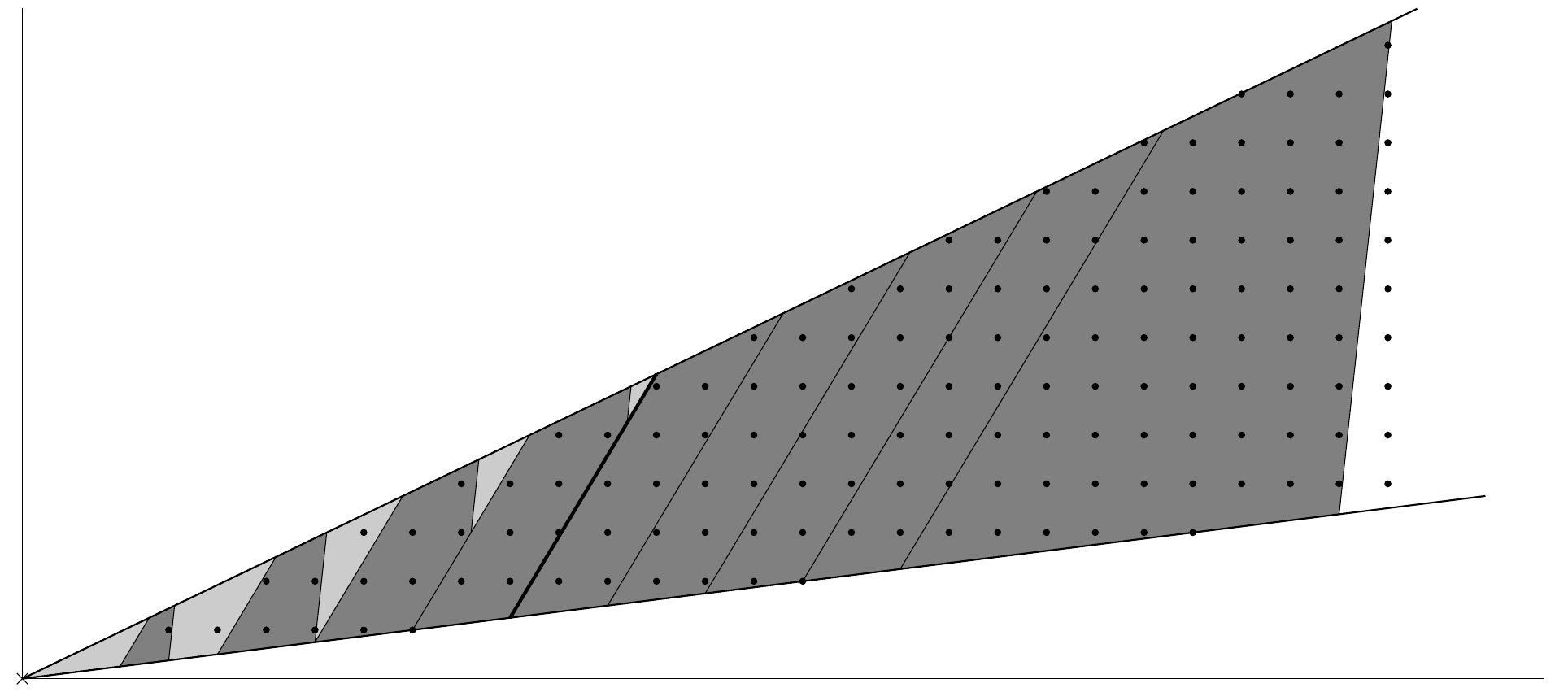} \\ \hline
\end{tabular}
\caption{The affine polygonal semigroup $\FraP$ associated to the polygon $\{(2, 0.25),( 3, 0.375),( 2.6, 1.25),( 3.12, 1.5)\}$. Since $\int(\FraC)\setminus \int(\FraP)=\{(4,1),(7,3)\}$,   $\int (\FraC) = \int(\overline{\FraP}).$ Besides, $\overline{\FraP}\cap \tau_j$ is generated by only one element for $j=1,2$, and therefore $\FraP$ is Buchsbaum.}\label{caso_2_segmentos}
\end{center}
\end{figure}
\item Let us suppose that $F\cap \tau_1$ is a point $P_1$ and $F\cap \tau_2$ is a point $P_d$.
If $P\in (\int (\FraC)\cap (\mathcal{B}_1\cup \mathcal{B}_2))\setminus \int(\FraP)$, the element $P+n_1$ or $P+n_2$ does not belong to $\FraP$ and thus $P\not\in\overline{\FraP}$. This implies that $\FraP\cap (\mathcal{B}_1\cup \mathcal{B}_2)=\overline{\FraP}\cap (\mathcal{B}_1\cup \mathcal{B}_2)$.
Let $j\in\N$ be
such that $j\overline{P_{1}P_{t}} \cap (j+1)\overline{P_{1}P_{2}}=\{V_1\}$ and
let $t\in\N$ satisfying $tP_1=n_1$.
For every  $r,k\in\Z_{\geq}$
 there exists $h\in\{0,\dots,t-1\}$ such that
$(\stackrel{\circ}{T_1}+(r+j) P_1)\cap \N^2
=(\stackrel{\circ}{T_1}+(h+j)P_1)\cap \N^2+kn_1$.
A similar construction must be done for $\mathcal{B}_2$ proceeding  similarly with the triangle $T_2$.
So to compare $\int (\FraC)\cap (\mathcal{B}_1\cup \mathcal{B}_2)$ with $\int(\overline{\FraP})\cap (\mathcal{B}_1\cup \mathcal{B}_2)$ it is only necessary to check if there are nonnegative integer points in the sets $\stackrel{\circ}{T_1}+(h+j)P_1$ (with $h\in\{0,\dots,t-1\}$) and, analogously, in some translations of $\stackrel{\circ}{T_2}$ in the direction of $P_d$. Since $\Upsilon_1$ and $\Upsilon_2$ are included in two parallelograms,  $\Upsilon_1\cup\Upsilon_2$ is a finite set and therefore $(\int (\FraC)\cap (\Upsilon_1\cup\Upsilon_2))\setminus \int (\overline{\FraP})$ can be computed.

 In order to compute  $(\int(\FraC)\cap \Upsilon)\setminus\int(\overline{\FraP})$, just take $j\in \N$ the least integer such that both sets $j\overline{P_1P_t}\cap (j+1)\overline{P_1P_2}=\{V\}$ and $j\overline{P_dP_{d+1}}\cap (j+1)\overline{P_dP_{d-1}}=\{V'\}$ are formed by only one point and let $T$ be the triangle with vertex set $\{Q,V,V'\}$. By construction, the sets $(\int(\FraC)\cap \Upsilon)\setminus T$, $(\int(\FraP)\cap \Upsilon)\setminus T$ and $(\int(\overline{\FraP})\cap \Upsilon)\setminus T$ are equal. Therefore $\int(\FraC)\cap \Upsilon=\int(\overline{\FraP})\cap \Upsilon$ if and only if the finite sets $\int(\FraC)\cap T$ and $\{a\in \int(\FraP)\cap T | a+n_i\in\FraP,\, \forall i=1,\ldots ,m\} $ are equal. This case is illustrate in Example \ref{ejemplo_poligono_buchsbaum} (see Figure \ref{caso_2_puntos}).

\item If $F\cap \tau_1=\{P_1\}$ and $F\cap \tau_2$ is a segment $\overline{P_{d-1}P_d}, $
    to compare the sets $\int(\FraC)\setminus \Upsilon$ and
    $\int(\overline{\FraP})\setminus \Upsilon$, just proceed as in the second case with the sets  $\mathcal{B}_1$ and $\Upsilon_1$.
Let now $j\in \N$  be the least integer such that $j\overline{P_1P_t}\cap (j+1)\overline{P_1P_2}$ is a point $V$ and $j\overline{P_{d-1}P_d}\cap (j+1)\overline{P_{d-1}P_d}\neq \emptyset ,$ and let $T$ be the triangle with vertex set $\{Q,V,jP_{d-1}\}.$ Then $\int(\FraC)\cap \Upsilon=\int(\overline{\FraP})\cap \Upsilon$ if and only if the finite sets $\int(\FraC)\cap T$ and $\int(\overline{\FraP})\cap T$ are equal.

\item Finally, the case $F\cap \tau_2$ is a point and $F\cap \tau_1$ is a segment is analogous to the above case.
\end{enumerate}

In any case, all the necessary sets to compare $\int(\FraC)$ with $\int (\overline{\FraP})$ are finite and they can be obtained algorithmically. Besides, the conditions
 $\Upsilon'=\emptyset$ and $\Upsilon\subset\overline{\FraP}$ can be checked
 algorithmically and
 "$\overline{\FraP}\cap \tau_j$ is generated by only one element" can be tested in a similar way to the case of circle semigroup.

\begin{example}\label{ejemplo_poligono_buchsbaum}
Let $F$ be the polygon determined by the rational points $\{( 3.6, 1.8),( 3.6, 0.6),( 3.3, 1.05),( 4.2, 1.5),( 4.14, 0.99)\}$ and $\FraP$ its associated affine convex polygonal semigroup (the dark grey region in Figure \ref{caso_2_puntos}).
The minimal system of generators of $\FraP$ can be computed
with the function {\tt PolygonalSG} (see \cite{programa_poligonos}),
\begin{verbatim}
In[1]:= PolygonalSG[{{3.6,1.8},{3.6,0.6},{3.3,1.05},
                            {4.2,1.5},{4.14,0.99}}]
Out[1]= {{4,1},{7,2},{7,3},{8,3},{10,3},{11,2},{11,5},{14,3},
            {18,3},{18,9},{20,8},{23,10}}
\end{verbatim}
We obtain that $\FraP$ is minimally generated by
\[\begin{multlined}
G=\{(18,9),(18,3),(4,1),(20,8),(23,10),(8,3),\\(11,5),(11,2),(10,3),(14,3),(7,2),(7,3)\},
\end{multlined}
\]
following the notation of the above sections $n_1=(18,9)$ and $n_2=(18,3)$.

Using basic tools of Linear Algebra we compute the sets $\Upsilon_1$, $\Upsilon_2$, the triangle $T$ and the necessary translations of $T_1$ and $T_2$ (the above sets are needed to check the conditions of Theorem \ref{theorem_polygon_Buchs}). Those translations are the lighter grey triangles in Figure \ref{caso_2_puntos}, $(\Upsilon_1\cup \Upsilon_2)\setminus\FraP$ is the region in middle light grey and $(\int(\FraC)\setminus \int(\FraP))\cap\Upsilon=(T\cap \N^2)\setminus \FraP=\{(13,4)\}$.
Since $(13,4)$ does not belong to $\FraP$, by Corollary \ref{C-M},
the semigroup $\FraP$ is not Cohen-Macaulay.
We also have $(13,4)+n\in\FraP$ for all $n\in G$. This can be checked with the function {\tt BelongToSG} of \cite{programa_poligonos}. For $n_1=(18,9)$, we obtain
\begin{verbatim}
In[2]:= BelongToSG[{13,4}+{18,9},{{3.6,1.8},{3.6,0.6},
                            {3.3,1.05},{4.2,1.5},{4.14,0.99}}]
Out[2]= True
\end{verbatim}
Thus $(13,4)\in \overline{\FraP}$ and therefore $\Upsilon\subset \overline{\FraP}.$

The set $(\int(\FraC)\setminus \int(\FraP))\cap (\Upsilon_1\cup \Upsilon_2)$ is equal to
\[\begin{multlined}
D=\{(3,1),(5,1),(5,2),(6,2), (9,2), (10,2),(9,3),(13,3),\\(16,3),(17,3),(9,4),(10,4), (17,4),(12,5),(13,5),(13,6)\},
\end{multlined}\]
but none of these points are in $\overline{\FraP}$. Besides, for all $a\in D,$ $a+n_1'$ or $a+n_2'$ does not belong to $\overline{\FraP}.$ Therefore $\Upsilon'$ is the empty set.
By Theorem \ref{theorem_polygon_Buchs},
we conclude that $\FraP$ is a non-Cohen-Macaulay Buchsbaum affine semigroup.

Using the function {\tt PSGIsBuchsbaumQ} of \cite{programa_poligonos}, the above Buchsbaumness problem can be solved in less than a second (all the examples of this work have been done in an Intel Core i7 with 16 GB of main memory),
\begin{verbatim}
In[3]:= AbsoluteTiming[PSGIsBuchsbaumQ[{{3.6,1.8},{3.6,0.6},
                            {3.3,1.05},{4.2,1.5},{4.14,0.99}}]]
Out[3]= {0.734412, True}
\end{verbatim}
\begin{figure}[h]
    \begin{center}
\begin{tabular}{|c|}\hline
\includegraphics[scale=0.45]{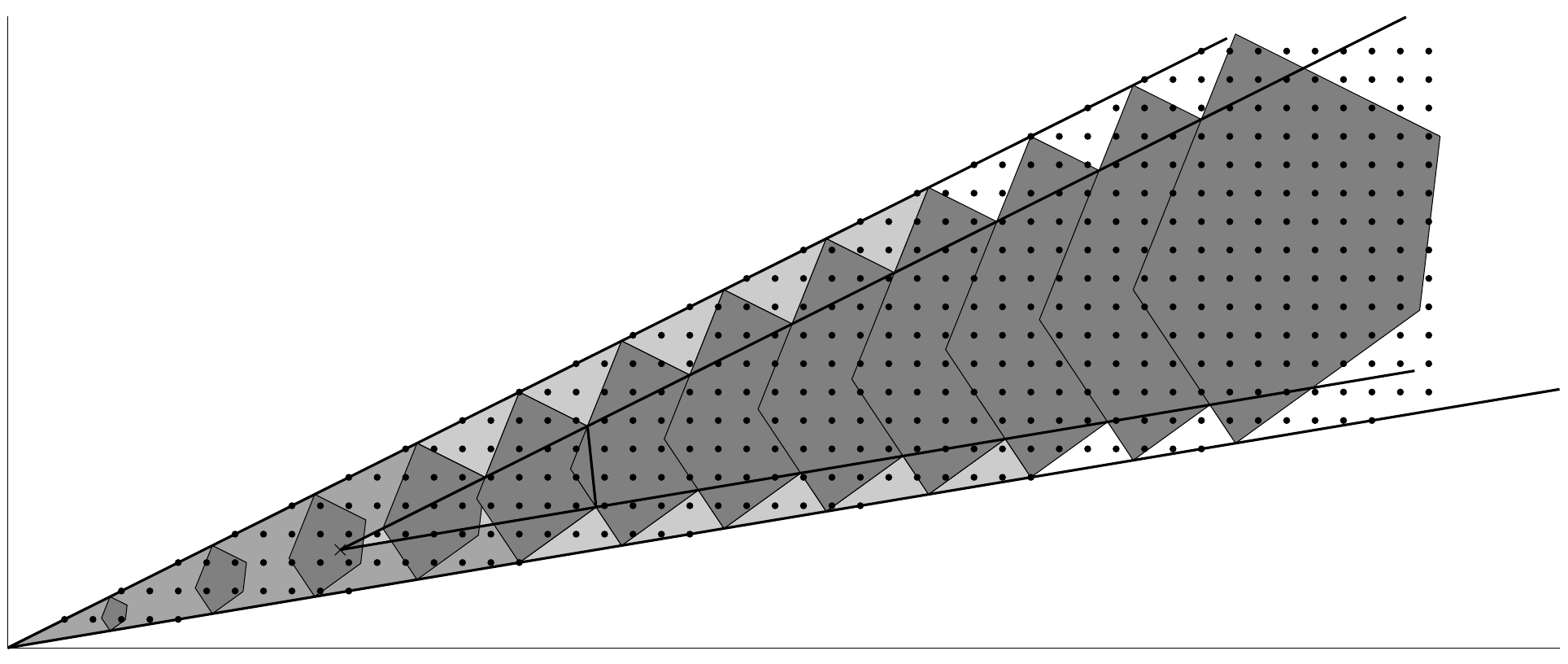} \\ \hline
\end{tabular}
\caption{The affine polygonal semigroup $\FraP$ associated to the polygon $\{( 3.6, 1.8),( 3.6, 0.6),( 3.3, 1.05),( 4.2, 1.5),( 4.14, 0.99)\}$.}\label{caso_2_puntos}
\end{center}
\end{figure}
The return value is \verb={0.734412, True}= where \verb=0.734412= are the seconds required for this computation and \verb=True= is the answer to the Buchsbaumness question.

If we use the method of Theorem 9 in \cite{RosalesBuchs}, it is necessary to compute the intersection of the Ap\'{e}ry set of $n_1$ and the Ap\'{e}ry set of $n_2$
by checking if $2 \times 7\, 771\, 556\, 800\, 000$ elements belong to $\FraP$; this causes such method to be inefficient.

\end{example}

As indicated before, the problem of determining whether an element belongs to a convex polygonal semigroup is straightforward; this implies a reduction of the time of computation.

In  Example \ref{ejemplo_poligono_buchsbaum}, it is used only Elementary Algebra, but Buchsbaum semigroups can be generated using an even simpler approach. The following results provide two user-friendly properties which allow us to obtain easily Buchsbaum rings.

\begin{corollary}\label{triangulos_C-M}
Every affine convex polygonal semigroup associated to a triangle with rational vertices is Buchsbaum.
\end{corollary}

\begin{proof}
Note that if $F$ is a triangle, $\FraP$ and $\overline{\FraP}$ are equal. Corollary 12 in \cite{convex_CM_Go} proves that every affine convex polygonal semigroup associated to a triangle with rational vertices is Cohen-Macaulay. Thus,  $\overline{\FraP}=\FraP$ is Cohen-Macaulay and therefore $\FraP$ is Buchsbaum (Theorem \ref{RosalesBuchs}).
\end{proof}

\begin{corollary}\label{rombos_C-M}
Let $F$ be a convex polygon with vertices $P_1,\ldots ,P_4 \in \Q_{\geq} ^2$ and let  $\FraP$ be its associated affine convex polygonal  semigroup. If $P_1\in \FraP\cap \tau_1,$ $P_4\in \FraP\cap \tau_2$ and the points $O,$ $P_2$ and $P_3$ are aligned, $\FraP$ is Buchsbaum.
\end{corollary}

\begin{proof}
Let $\FraC_1$ be the positive integer cone delimited by the ray $\tau_1$ and the line $OP_2$, and let $\FraC_2$ be the cone delimited by the ray $\tau_2$ and the line $OP_2$. Trivially  $\FraC=L_{\Q_\geq}(F)\cap \N^2$ is the union of $\FraC_1$ and $\FraC_2$, and the semigroup $\FraP$ is the union of the affine convex polygonal semigroups, $\FraP_1$ and $\FraP_2$, associated to the triangles with vertex sets $\{P_1,P_2,P_3\}$ and $\{P_2,P_3,P_4\}$, respectively.
With that decomposition of the affine convex polygonal semigroup $\FraP$ and from the hypothesis, we can assert that $\FraP$ is equal to $\overline{\FraP}$, that $\Upsilon\subset \FraP$ and that $\FraP\cap \tau_1$ and $\FraP\cap \tau_2$  are semigroups generated by only one element (Figure \ref{figura_rombo} illustrates this situation).
\begin{figure}[h]
    \begin{center}
\begin{tabular}{|c|}\hline
\includegraphics[scale=0.5]{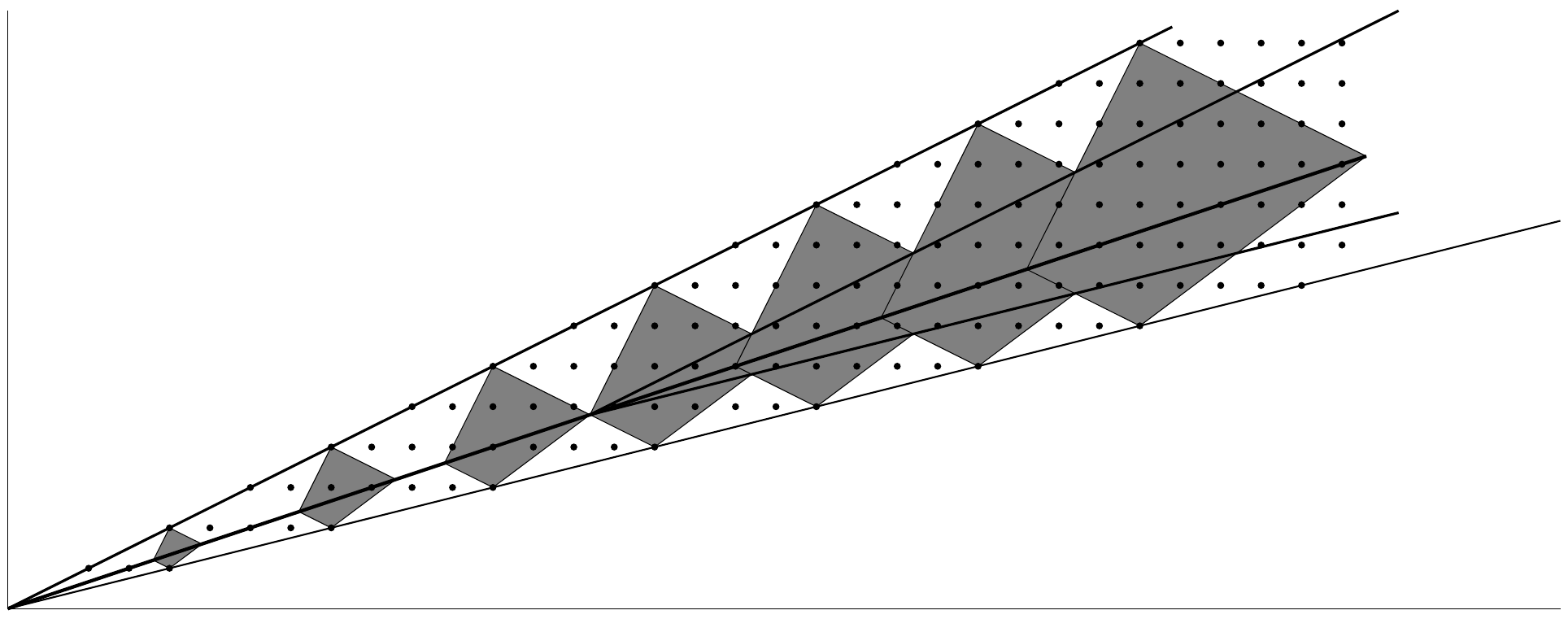} \\ \hline
\end{tabular}
\caption{The affine polygonal semigroup $\FraP$ associated to the polygon $\{(3.6, 1.2), (4.8, 1.6), (4, 2), (4, 1)\}$.}\label{figura_rombo}
\end{center}
\end{figure}
Under such conditions, let $a$ be an element belonging to $\FraC\setminus\FraP.$ Note that if $a\in \FraC_1\setminus \FraP_1$ then $a+n_1\notin \FraP,$ otherwise, if $a\in \FraC_2\setminus \FraP_2$ then $a+n_2\notin \FraP.$ In any case, $a+n_1$ or $a+n_2$ does not belong to $\FraP$. Thus $\overline{\FraP}$ ($=\FraP$) is Cohen-Macaulay (Corollary \ref{C-M}) and then $\FraP$ is Buchsbaum.
\end{proof}

\end{document}